\documentclass[12pt,reqno]{amsart}

\numberwithin{equation}{section}

\newtheorem{teo}{Theorem }[section]

\newtheorem{lem}[teo]{Lemma}

\newtheorem{rem}{Remark}
\def \R{ \mathbf{R}}

\begin{document}

%%%%%%%%%%%%%%%%%%%%%%%%%%%

\title[2D BENNEY-LIN EQUATION
]
      { Global solutions for a  2D BENNEY-LIN type  EQUATION posed ON RECTANGLES and on a half-strip}
\author[ N.~A. Larkin]{Nikolai A. Larkin$^{\dag}$}

\address
{
Departamento de Matem\'atica, Universidade Estadual
de Maring\'a, Av. Colombo 5790: Ag\^encia UEM, 87020-900, Maring\'a, PR, Brazil
}

\email{$^{\dag}$nlarkine@uem.br}

\thanks{}

\begin{abstract}
We formulate on rectangles and on the right horizontal half-strip  initial-boundary value problems for
a two-dimensional Benney-Lin type equation.
 Existence and uniqueness of a regular solution as well as the exponential decay rate for the  norm
of a regular solution have been established.
\end{abstract}

\maketitle

\section{\bf Introduction}

We are concerned with  initial-boundary value problems (IBVP) posed
on rectangles and the right horizontal half-strip for the 2D Benney-Lin-Kawahara equation (BLK)
$$ u_ t+\Delta^2u +\gamma\Delta u+\Delta u_x +uu_x + \eta\partial^5_x u=0    \eqno(1.1)$$
 which is a two-dimensional analog of the well-known Benney-Lin
equation, \cite{benney,lin,biagioni,phys,fam3,kawa,ax},
$$
	 u_t+\eta D^5_x u+\beta D^4_x u+\alpha D^3_x u+ \gamma D^2_x u +uu_x=0. \eqno(1.2)
$$ 
In turn, (1.1) includes as special cases the KdV, Kuramoto-Sivashinsky, Kawahara equations., \cite{kuramoto,sivash,cousin,familark,topper}.
 The theory of the Cauchy problem for (1.2) and other dispersive
equations has been extensively studied and is
considerably advanced today
\cite{biagioni, cui,kato,temam1}. In
recent years, results on IBVPs  for dispersive equations both in
bounded and unbounded domains have appeared
\cite{bubnov,chile,doronin3,familark,ax,larkin,larkin1,larkin2,pastor2,saut4}. It was
discovered in \cite{larkin,marcio} that the KdV and Kawahara
equations have an implicit internal dissipation. This allowed to
prove  exponential decay of small solutions in bounded domains
without adding any artificial damping term. Later, this effect was
proved for a wide class of dispersive equations of any odd order
with one space variable \cite{familark}.
 Publications on dispersive multy-dimensional equations of a higher
order (such as the KZK equation) appeared quite recently and were
concerned with the existence of weak solutions, \cite{fam3}, and
physical motivation \cite{phys}.\par Our work here was motivated by \cite{topper} where stability of solutions due to presense of higher order terms has been studied. Coefficient $\gamma$ in (1.1) may be positive or negative depending on wavenumbers of long waves considered. The condition $\gamma>0$ is considered to be instability condition, hence to guarantee stability of the system, we must take into account the fourth-order terms. On the other hand, the condition $\gamma<0$ guarantees stability  while the fourth-order terms give an additional damping effect. Differently from \cite {topper}, we added the fifth-order Kawahara term putting $\eta=-1$.
\par We study (1.1) on rectangles
$$
D=\left\{(x,y) \in \mathbb{R}^{2}: \quad x \in (0,L), \quad y \in (0,B), \quad \partial D
\right\},
$$
where $\partial D$ is a boundary of $D$.
For $\gamma>0$, we establish  the existence and uniqueness of global regular solutions for $t\in(0,T),$  where $T$ is an arbitrary positive number, without smallness conditions for the initial data. 
 More precisely, we formulate in Section 2
IBVP (2.1)-(2.3). In order to demonstrate existence of global
regular solutions, we exploit the Faedo-Galerkin method. Estimates,
independent of the parameter of approximations $N$, permit us to
establish the existence of regular solutions for the original
problem (2.1)-(2.3). We prove these estimates in Section 3. \\

In Section 4, we pass to the limit as $N \to \infty$ and obtain a
global regular solution of (2.1)-(2.3). In Section 5, we prove
uniqueness of a regular solution. Finally, in Section 6, we
establish the exponential decay rate of small
solutions both for $\gamma>0$ and for $\gamma\leq0.$ Section 7 contains Conclusions.

\section{Notations and Auxiliary Facts}

Let $B,L$ be  positive numbers and $x \in (0,L); y\in (0,B); D=D(x,y)\in \mathbf{R^2}.$ We use the standard notations of Sobolev spaces $W^{k,p}$, $L^p$ and $H^k$ for functions and the following notations for the norms \cite{Adams, Brezis}:
$$\R^+=\{t\in\R^1;\;t> 0\},\;\;\| f \|^2 =(f,f)= \int_0^L\int_0^B | f |^2dxdy,$$
$$\|f\|^p_{L^p}=\int_0^L\int_0^B|f|^pdxdy,\;\;\| f \|_{W^{k,p}} = \sum_{0 \leq \alpha \leq k} \|D^\alpha_x f \|_{L^p},\;$$$$ D^\alpha_x=\frac {d^{\alpha}}{dx^{\alpha}},\;D^\alpha_y=\frac {d^{\alpha}}{dy^{\alpha}},\;\;D_x=D^1_x,\;D^0_xu=u.$$
We will use also the standad notations: $$D_xf=f_x,\;D_x^2 f=f_{xx},\;\frac{\partial}{\partial t}f=f_t,\;D_yf=f_y.$$
When $p = 2$, $W^{k,p} = H^k$ is a Hilbert space with the scalar product 
$$((u,v))_{H^k}=\sum_{|j|\leq k}(D^ju,D^jv),\;\|u\|_{\infty}=
\|u\|_{L^{\infty}(D)}=ess \sup_{(D)}|u(x)|.$$
We use the notation $H_0^k(D)$ to represent the closure of $C_0^\infty(D)$, the set of all $C^\infty$ functions with compact support in $(D)$, with respect to the norm of $H^k(D)$.

\begin{lem}[Steklov's Inequality \cite{steklov}] Let $v \in H^1_0(0,L).$ Then
	\begin{equation*}\label{Estek}
		\frac {\pi^2}{L^2}\|v\|^2(t) \leq \|v_x\|^2(t).
	\end{equation*}
\end{lem}

\begin{lem}
	[Differential form of the Gronwall Inequality]\label{gronwall} Let $I = [t_0,t_1]$. Suppose that functions $A,B:I\to \R$ are integrable and a function $A(t)$ may be of any sign. Let $u:I\to \R$ be a differentiable function satisfying
	\begin{equation*}
		u_t (t) \leq A(t) u(t) + B(t),\text{ for }t \in I\text{ and } \,\, u(t_0) = u_0,
	\end{equation*}
	then
	
	\begin{equation*}u(t) \leq u_0 e^{ \int_{t_0}^t A(\tau)\, d\tau } + \int_{t_0}^t e^{\int_{t_0}^s A(r) \, dr} B(s)\, ds.\end{equation*}.
\end{lem}
\begin{proof}
	Multiply the last inequality by the integrating factor $e^{\int_{t_0}^{s} A(r)\, dr}$ and integrate from $t_0$ to $t$.
\end{proof}

The next Lemmas will be used in  estimates:\\
\begin{lem}
 (See \cite{niren}, p. 125.)
Suppose $u$ and $D^mu$, $m\in\mathbb{N},$ belong to $L^2(0,L)$. Then for the derivatives $D^iu$, $0\leq i<m$, the following inequality holds: 

$$	\|D^iu\|\leq A_1\|D^mu\|^{\frac{i}{m}}\|u\|^{1-\frac{i}{m}}+A_2\|u\|,$$

where $A_1$, $A_2$ are constants depending only on $L$, $m$, $i$.
\end{lem}

\begin{lem}[See: \cite{lady2} ]
	
i) For all $u \in H^1_0(D)$
\begin{equation*} \|u\|^2_{L^4(D)} \leq 2 \|u\|_{L^2(D)}\|\nabla u\|_{L^2(D)}.
\end{equation*}
\qquad \qquad \qquad \qquad ii) For all $u \in H^1(D)$
\begin{equation*} {\|u\|}_{L^4(D)}^2 \leq C_D {\|u\|}_{L^2(D)}{\|u\|}_{H^1(D)}, \label{p2}
\end{equation*}
where the constant $C_D$ depends on a way of continuation of $u \in
H^1(D)$ as $ \tilde{u}(\mathbb{R}^2)$ such that $\tilde{u}(D)=u(D).$	
	
\end{lem} 

\begin{lem} [see \cite{ax}, Lemma 5.] 
	Let $f(t)$ be a continuous  positive function such that $f'(t)$ is a measurable, integrable function and
	
	$$f'(t) + (\alpha - k f^n(t)) f(t) \leq 0,\;t>0,\;n\in \mathbf{N}; \eqno(3)$$
	$$ \alpha - k f^n(0)> 0,\;k>0. \eqno(4).$$
	Then
	
	$$f(t) < f(0)$$ 	 for all $t > 0$.
\end{lem}
\section{Formulation of the problem }

Let $T,L,B$ be arbitrary real positive numbers;
\begin{eqnarray} && D= \left\{(x,y) \in \mathbb{R}^{2}: \quad x\in(0,L), \quad y \in (0,B) \right\}; \nonumber \\
&& Q_t= D \times (0,t), \quad t \in (0,T). \nonumber
\end{eqnarray}
Consider in $Q_T$ the following IBVP:

\begin{align}
	& Lu \equiv u_t + \Delta^2 u+ \Delta u+ \Delta u_x +uu_x-\partial^5_x u=0 \quad \textrm{in} \quad Q_t;  \\
	& u|_{\partial D}=u_{yy}(x,0)=u_{yy}(x,B)= u_x(0,y)\notag\\&=u_x(L,y)=u_{xx}(L,y)=0; \\
	& u(x,y,0)=u_0(x,y), \quad (x,y) \in D. 
\end{align}

Here 
 ${\Delta}= \partial_x^2 + \partial_y^2$. We adopt the usual
 notations $H^k$ for $L^2$-based Sobolev spaces;\,
 $\| \cdot \|$ and $(\cdot , \cdot )$ denote the norm and the scalar product in $L^2(D)$, ${|\nabla u |}^2=u_x^2+u_y^2$.

\section{Existence Theorem}

\begin{teo}\label{T1} Let $T,B,L$ be arbitrary real positive numbers. Given $u_0(x,y)$  such that
\begin{eqnarray*}
&& u_0 \in H^4(D), \quad   \partial_x^5 u_0\;\in L^2(D), \nonumber \\
&& u_0(0,y)=u_{0x}(0,y)=u_0(L,y)=u_{0x}(L,y)=u_{0xx}(L,y)\\&&=u_0(x,0)=u_0(x,B)=u_{0yy}(x,0)=u_{0yy}(x,B)=0, \nonumber\\&&
J_w=\int_D \big[|\Delta^2 u_0|^2+|\Delta u_0|^2+u^2_0 u^2_{0x}+|D^5_{0x}|^2\big]dxdy<\infty,\end{eqnarray*}

then there exists a unique regular solution of (3.1)-(3.3):
\begin{eqnarray*}
&& u \in L^{\infty} (0,T;H^2(D)) \cap L^2 (0,T;H^4(D)),  \\
&& \partial_x^5 u \;\in L^2(0,T;L^2(D)),
 u_t \in L^{\infty} (0,T;L^2(D)) \cap L^2(0,T;H^2(D)). \\
\end{eqnarray*}
\end{teo}

\begin{proof}

{\bf Approximate Solutions.}

To prove the existence part of this theorem, we
use the Faedo-Galerkin Method as follows:\\ for all $N$ natural, we
define an approximate solution of (3.1)-(3.3) in the
form
\begin{equation}
u^N(x,y,t)=\sum^N_{j=1}\omega_j(y)g_j(x,t), \label{2.1}
\end{equation}
where $\omega_j(y)$ are orthonormal in $L^2(0,L)$ eigenfunctions of
the following  problem:
\begin{eqnarray}
&&- D^2_y\omega=\lambda_j \omega_j(y),\;y\in(0,B);\;\omega_j(0)=\omega_j(B)=0
\end{eqnarray}
and $g_j(x,t)$ are solutions to the following initial- boundary value
problem for the system of N generalized KdV equations:
\begin{eqnarray}
&&\frac{\partial}{\partial t}g_j(x,t)+\sum_{l,k=1}^N
a_{lkj}g_l(x,t)g_{kx}(x,t)+(\Delta u^N,\omega_j)\nonumber \\
&&+(\Delta^2u^N,\omega_j) -\partial^5_x g_j(x,t)+(\Delta u^N_x,\omega_j)=0, \nonumber  \\
&& g_j(0,t)=g_{jx}(0,t)=0, \;g_j(x,0)=u_{0j}(x)\label{eN},
\end{eqnarray}
where
\begin{eqnarray}
&&a_{klj}=\int_0^L \omega_k(y)\omega_l(y)\omega_j(y)\,dy,
\;j,k,l=1,...,N;\nonumber\\
&&u_{0j}(x)=\int_0^L u_0(x,y)\omega_j(y) dy.\nonumber
\end{eqnarray}
 Solvability of (4.3) (at least local in t) follows from
\cite{kuvsh}. Hence, our goal is to prove necessary a
priori estimates, uniform in  $N$, which will permit us to pass to
the limit in (4.1) as $N\to \infty$ and to establish the
existence result. We assume first that a function $u_0$ is
sufficiently smooth to ensure  calculations.  Exact conditions for
$u_0$ will follow from a priori estimates for $u^N$ independent of
$N$ and usual compactness arguments.

{\bf ESTIMATE I.} Multiplying the $j$-equation of (4.3) by
$g_j(x,t)$, summing over $j=1,..,N$ and integrating the result with
respect to $x$ over $R^+$, we obtain
\begin{eqnarray}
&&\frac{1}{2}\frac{d}{dt}\|u^N\|^2(t)+(|u^N|^2,
u_x^N)(t)+(u^N,\partial^3_x u^N)(t)+(\Delta u^N,u^N)(t)\nonumber\\
&&+(\Delta^2 u^N,u^N)(t)-(u^N,\partial^5_x u^N)(t)+(u^N,\partial^2_y u_x^N)(t)=0.\nonumber
\end{eqnarray}
In our calculations we will drop the index $N$ where this is not
ambiguous. Integrating by parts the last equality, we get

\begin{equation}\frac{d}{dt}\|u\|^2(t)+\|\Delta u\|^2(t)+\int^L_0 u_{xx}^2(0,y,t)\,dy\leq\|u\|^2(t).\end{equation}
Applying Lemma 2.2 to the  inequality
$$\frac{d}{dt}\|u\|^2(t)\leq\|u\|^2(t),$$
we get
\begin{equation}
	\|u^N\|^2(t)\leq e^T \|u_0\|^2. \label{E1}
\end{equation}
Returning to (4.4), we find that for $N$ sufficiently large and $\forall t\in(0,T)$
\begin{equation}\|u^N\|^2(t)+\int_0^t[\|\Delta u^N\|^2(\tau)+\int^L_0 |u^N_{xx}|^2(0,y,\tau)\,dy]d\tau\leq e^T\|u_0\|^2.\end{equation}

 \begin{lem} [\cite{larkin1} Lemma 3.1.]
	Let $f\in H^4(D)\cap H^1_0(D)$ and $  f|_{\partial D}$ satisfies (4.2). Then
	\begin{align}
		&\|\nabla f\|^2\geq a\|f\|^2,\;\;a^2\|f\|^2\leq \|\Delta f\|^2,\;\;a\|\nabla f\|^2\leq \|\Delta f\|^2,\\
		&a^2\|\Delta f\|^2\leq \|\Delta^2 f\|^2,\;\;\text{where} \; a=\frac{\pi^2}{L^2} +\frac{\pi^2}{B^2}=\frac{\pi^2(B^2+L^2)}{L^2B^2}.
	\end{align}	
\end{lem}

\begin{lem}Let $f\in H^2(D)\cap H^1_0(D).$ Then
	\begin{equation}
		\sup_D u^2(x,y)\leq C_s \|\Delta u\|^2,\;\text{where}\; C_s=1+\frac{1}{a}+\frac{1}{a^2}.
	\end{equation}
	\begin{proof} Dropping variable $t$ and making use of Lemma 4.2, we can write
		\begin{align*}
			\sup_D u^2(x,y) \leq \sup_{x\in(0,L)}\sup_{y\in(0,B)}|\int_0^y \frac{\partial}{\partial s}u^2(s,x)ds|\\
			\leq \sup_{x\in(0,L)}\int_0^B\{u_y^2(x,y)+u^2(x,y)\}dy\\
			\leq \int_0^L\int_0^B\{|2u_{xy}(x,y)u_y(x,y)+2u_x(x,y)u(x,y)|\}dxdy\\
			\leq \|u_{xy}\|^2+\|\nabla u\|^2+\|u\|^2\leq (1+\frac{1}{a}+\frac{1}{a^2})\|\Delta u\|^2.
		\end{align*}
	\end{proof}
\end{lem}
		The proof of Lemma 4.3
		 is complete.	\\

{\bf ESTIMATE III.} Taking into account the structure of
$u^N(x,y,t)$, consider the scalar product
\begin{equation}
-2(\partial^2_y u^N,[u^N_t+\Delta^2u+\Delta u^N+u^Nu_x^N+\Delta u^N_x
-\partial^5_x u^N])(t)=0. \nonumber
\end{equation}
Acting as by proving Estimate II and dropping the index $N$, we come
to the following equality:
\begin{eqnarray}
&&\frac{d}{dt}\|u^2_y\|(t)+\int^L_0 u_{xxy}^2(0,y,t)\,dy
-2(\Delta^2 u,D^2_y u)(t)\nonumber\\&&=2(u_{yy}, uu_x)(t). \label{e31}
\end{eqnarray}

We estimate
$$I_1=-2(\Delta^2 u,D^2_y u)(t)=2([|D^3_y u|^2+2u^2_{xyy}+u^2_{xxy}])(t),$$

$$ I_2=2(u_{yy},uu_x)(t)=(u_{yy},(u^2)_x)(t)\leq +2|(u_{xy},uu_y)|(t)$$
$$\leq +\epsilon\|u_{xy}\|^2(t)+\frac{sup_D u^2(x,y)}{\epsilon}\|u_y\|^2(t)$$
$$\leq  \epsilon\|u_{xy}\|^2(t)+\Big[2kC^{1/2}_s\|\Delta u\|(t)+\frac{C_s}{\epsilon}\Big]\|u_y\|^2(t)$$
$$\leq \epsilon\|u_{xy}\|^2(t)+\frac{C}{\epsilon}[1+\|\Delta u\|^2(t)]\|u_y\|^2(t).$$

Substituting $I_1-I_2$ into (\ref{e31}), taking $\epsilon
>0$ sufficiently small and using (4.5)-(4.6), we come to the
inequality
\begin{eqnarray}
&& \dfrac{d}{dt}{\|u_y\|}^2(t)+\displaystyle\int_{0}^{L}u^2_{yxx}(0,y,t)\,dy +2(|D^3_y u|^2+2u^2_{xyy}+u^2_{xxy}])(t)\nonumber\\&&\leq  \epsilon\|u_{xy}\|^2(t)+\frac{C}{\epsilon}[1+\|\Delta u\|^2(t)]\|u_y\|^2(t).
 \label{e32} \end{eqnarray} Taking $\epsilon$ sufficiently small and making use of the Gronwall lemma
and Estimates I, II, we find
\begin{eqnarray*}
\|u_y\|^2(t)& \leq & \|u_{0y}\|^2e^{C(\|u_0\|,T)\displaystyle\int_{0}^{t}[1+\|\Delta u \|^2(\tau)]d\tau.}  \\
& \leq & \| u_{0y}\|^2e^{C( \|u_0\|,T)}
\leq C(T,\|u_{0}\|_{H^1_0}).
\end{eqnarray*}
Integrating (\ref{e32}) over $(0,t)$ gives
\begin{eqnarray}
& \|u^N_y\|^2(t)
+ \displaystyle\int_{0}^{t}2([\|D^3_y u^N\|^2+2\|u^N_{xyy}\|^2+\|u^N_{xxy}\|^2])(\tau)\nonumber\\&+[\|u^N_{xy}\|^2(\tau)+\|u^N_{yy}\|^2(\tau)]\, d\tau \nonumber \\
&+\displaystyle\int_{0}^{t}\int^L_0 |u^N_{xxy}(0,y,\tau)|^2\,dy
d\tau  \leq C(T,\|u_0\|)\|u_{0y}\|^2). \label{E3}
\end{eqnarray}

{\bf ESTIMATE IV.}  Dropping the index $N$, transform the scalar
product
$$2(\partial^4_y u^N,[u^N_t+\Delta^2 u^N+\Delta u^N+u^N u^N_x+\Delta u^N_x
-\partial^5_x u^N])(t)=0
$$
into the following equality:
\begin{eqnarray}
\dfrac{d}{dt}\|u_{yy}\|^2
(t)
&& +\displaystyle\int^L_0|\partial^2_y u_{xx}(0,y,t)|^2\,dy
++2(D^4_yu,\Delta^2u)(t)\nonumber\\
&& =-2(\Delta u,D^4_y u)(t)-2(D^4_y u,uu_x)(t). \label{e4}
\end{eqnarray}
We estimate
$$I=2(D^4_yu,\Delta^2u)(t)=2(D^4_yu,[D^4_y+2D^2_yD^2_x+D^4_x]u)(t)=I_1+I_2+I_3,$$
where
\begin{eqnarray*}
&&I_1=2(D^4_yu,D^4_yu)(t)=2\|D^4_yu\|^2)(t),\\
&&I_2=(D^4_yu,2D^2_yD^2_xu)(t)=4\|D^3_yu_x\|^2)(t),\\
&&I_3=2(D^4_yu,D^4_xu)(t)=2\|D^2_yD^2_xu\|^2)(t).
\end{eqnarray*}
Hence
\begin{eqnarray*}
	&&I=2[\|D^4_y\|^2+2\|D^3u_x\|^2+\|D^2_yD^2_xu\|^2](t).
\end{eqnarray*}

For any positive $\epsilon$, integrating by parts, we obtain
\begin{eqnarray*}
&&I_4=2(D^4_yu,uu_x)=-2(D^3_yu_x,uu_y)\leq\epsilon\|D^3_yu_x\|^2)\\&&+\frac{1}{\epsilon}(u^2,u^2_y)
\leq \epsilon\|D^3_yu_x\|^2)+\frac{1}{\epsilon}\sup_D u^2(x,y)\|u_y\|^2(t)\\
&&\leq \epsilon\|D^3_yu_x\|^2)+\frac{1}{\epsilon}C_s\|\Delta u\|^2(t)\|u_y\|^2(t).
\end{eqnarray*}

Taking  $\epsilon >0$ sufficiently small and substituting $I_1-I_4$
into (\ref{e4}), we obtain
\begin{eqnarray}
\dfrac{d}{dt}\|u_{yy}\|^2(t)+2([\|D^4_y\|^2+2\|D^3u_x\|^2+\|D^2_yD^2_xu\|^2])(t)\nonumber\\
+\displaystyle\int_{0}^{L}{|D_y^2u_{xx}(0,y,t)|}^2dy  \leq C(k)\left[ \|\Delta u \|^2(t)+1\right]\|u_{y}\|^2(t).
\end{eqnarray}\\
The previous estimates and the Gronwall lemma yield
\begin{eqnarray}
&& \|u^N_{yy}\|^2)(t)
+\displaystyle\int_{0}^{t}\{2([\|D^4_y u^N\|^2+2\|D^3_yu^N_x\|^2+\|D^2_yD^2_xu^N\|^2])(\tau)
 \,d\tau \nonumber \\
&&+\displaystyle\int_{0}^{t}\displaystyle\int_{0}^{L}{|D_y^2u^N_{xx}(0,y,\tau)|}^2dyd\tau\leq
C(\|u_{0y}\|^2+\|u_{0yy}\|^2), \label{E4}
\end{eqnarray}
where the constant $C$ does not depend on $N,L.$

{\bf ESTIMATE V.} To estimate $u^N_t$, we differentiate (\ref{eN})
with respect to $t$, multiply the $j$-equation of the resulting
system by $g_{jt}$, sum up over $j=1,...,N$ and integrate over
$D$. Calculations, similar to those exploited in Estimate II,
imply

\begin{eqnarray}
&&\frac{d}{dt}\|u_t\|^2(t)+\int_0^L u_{t xx}^2(0,y,t)\,dy +2\|\Delta u_t\|^2(t)\nonumber\\&&
+2(u_t,\Delta u_t)(t)=2([u u_x]_t,u_t)(t)=-(u_{xt},[u^2]_t)(t)-2(u_{xt},uu_t)(t)\nonumber\\&&\leq \epsilon\|u^2_{xt}\|(t)+[2ksup_D|u(x,t)|+\frac{1}{\epsilon}]\|u_t\|^2(t).
\label{5.1}
\end{eqnarray}

Taking  $\epsilon>0$ sufficiently
small and making use of Estimates I-IV and the Gronwall lemma, we
find

\begin{eqnarray*}
\|u_t\|^2(t)&  \leq &C(T,\|u_0\|)J_w.\nonumber
\end{eqnarray*}
Returning to (\ref{5.1}), we deduce
\begin{eqnarray}
&&\|u^N_t\|^2(t)+\int_0^t\int_0^L
|u^N_{xxs}(0,y,s)|^2\,dyds\nonumber\\
&&+\int_0^t\|\Delta u^N_s\|^2])(s)\,ds
\leq C(T)J_w, \label{Ev}
\end{eqnarray}
where 
\begin{equation}J_w=\int_D[|\Delta^2 u_0|^2+|\Delta u_0|^2+u^2_0 u^2_{0x}+|D^5_{0x}|^2]dxdy<\infty.\end{equation}

{\bf ESTIMATE VI.}

From (4.5),(4.6),(4.12),(4.15),(4.17), it follows that
\begin{eqnarray}
u^N,\;u^N_t\in L^{\infty}(0,T;L^2(D))\cap L^2(0,T;H^2_0(D));\\
D^2_y u^N,\;\nabla D^2_y u^N,\;\nabla D^3_y u^N,\;D^2_yD^2_x u^n\in L^2(0,T;(D))
\end{eqnarray}
and these inclusions do not depend on $N.$ Making use of (4.19)-(4.20), we find that
\begin{equation}
u^N\in L^{\infty}(0,T;H^2_0(D))
\end{equation}
uniformly on $N$.\\
Estimates (4.19)-(4.21) make it possible to pass to the limit as $N\to \infty$ in (4.1) and to prove thee existence of a weak solution of (3.1)-(3.3)   in the following sense:
\begin{eqnarray}
\int_0^T\{([u_t+D^4_y u+2D^2_xD^2_y u+\Delta u-D^2_y u],\phi)(t)\nonumber\\
+(D^2_x u,D^3_x \phi)(t)+(D^2_x u,D^2_x \phi)(t)-(D^2_x u,\phi_x)(t)\}dt.
\end{eqnarray}
Here, $u(x,y,t)$ satisfies the same inclusions (4.19)-(4.21) as $u^N$ and $\phi(x,y,t)$ is an arbitrary function from the class: $\phi\in L^{\infty}(0,T;L^2(D))\cap L^2(0,T;H^3_0(D)).$
Taking into account (4.19)-(4.21), for any $t_0\in (0,T)$ nonsingular, we can write

\begin{eqnarray}
([u_t+D^4_y u+2D^2_xD^2_y u+\Delta u-D^2_y u],\phi)(t_0)\nonumber\\
	+(D^2_x u,D^3_x \phi)(t_0)+(D^2_x u,D^2_x \phi)(t_0)-(D^2_x u,\phi_x)(t_0)=0.
\end{eqnarray} \\
In turn, for any $t\in (0,T)$ nonsingular, (4.23) may be rewritten as a distribution in $D$:
\begin{eqnarray}
	D^5_x u -D^4_x u-D^3_x u\nonumber\\
	= u_t+D^4_yu+2D^2_xD^2_yu+D^2_yu_x+\Delta u+uu_x\equiv f(x,y).
\end{eqnarray} 
Since $u$ is a limit of $u^N$ and satisfies (4.19)-(4.21), then $f\in L^2(0,T;L^2(D)).$
By Lemma 2.3,
\begin{eqnarray}
\|D^4_xu\|\leq C\|D^5_xu\|^{4/5}\|u\|^{1/5}\leq \delta\|D^5_xu\|+C_4(\delta)\|u\|,\nonumber\\
\|D^3_xu\|\leq C\|D^5_xu\|^{3/5}\|u\|^{2/5}\leq \delta\|D^5_xu\|+C_3(\delta)\|u\|,
\end{eqnarray}
where $\delta$ is an arbitrary positive number. Substituting (4.25) into (4.24), we get
$$\|D^5_xu\|(1-2\delta)\leq f+C(\delta)\|u\|.$$
Taking $\delta=\frac{1}{4},$ by (4.24),
\begin{equation}
\|D^5_xu\|, \|D^4_xu\|, \|D^3_xu\|\in L^2(0,T;L^2(D)).
\end{equation}
Returning to (4.24), we find that a weak soltuion of(3.1)-(3.3) is such that
\begin{equation}
u\in L^{\infty}(0,T;H^2_0(D)), \;\Delta^2u, D^5_xu\in L^2(0,T;L^2(D)).
\end{equation}
Hence, (4.22) can be rewritten as
\begin{equation}
	\int_0^t([u_{\tau}+\Delta^2u+\Delta u_x +\Delta u+uu_x],\phi)(\tau)d\tau =0,\;t\in(0,T),
\end{equation}
where $\phi(x,y,t)$ is an arbitrary function from $L^2(Q_t).$

\section{Uniqueness}
 Let $u_1$ and $u_2$ be distinct solutions of (3.1)-(3.3)
 and $z=u_1-u_2$. Then $z(x,y,t)$ satisfies
the following initial-boundary value problem:

\begin{eqnarray}
&& Lz= z_t +  \dfrac{1}{2}(u_1^2-u_2^2)_x + \Delta z_x -D_x^5 z\nonumber\\&& +\Delta z+\Delta^2z=0 \ \textrm{in} \ Q_t; \label{u1} \\
&& z(0,y,t)= z_x(0,y,t)=z_{x}(L,y,t)=z_{xx}(L,y,t)\nonumber\\&&=z(x,0,t)=z(x,B,t)=0,;\; \;  y \in (0,L), \; x>0, \; t>0; \label{u2}\\
&& z(x,y,0)=0, \quad (x,y) \in D. \label{u3}
\end{eqnarray}

From the scalar product
\begin{equation}
2(Lz, z)(t)=0, \label{eu}
\end{equation}
 acting in the same manner as by the proof of Estimate II, we obtain
\begin{eqnarray}
&&\frac{d}{dt}\|z\|^2(t)+2\|\Delta z\|^2(t)+\int_0^B z^2{xx}(0,y,\tau)dy +2(\Delta z,z)(t)\nonumber\\
&&=-([u_1^2-u_2^2]_x,z)(t).
\end{eqnarray} 
For an arbitrary positive $\delta$, we estimate
$$I_1=2(\Delta z,z)\leq \delta\|\Delta z\|^2+\frac{1}{\delta}\|z\|^2,$$
$$I_2=-([u_1^2-u_2^2]_x,z)=(z[u_1+u_2],z_x) \leq \frac{1}{2}\delta\|z_x\|^2+ \frac{1}{2\delta}\|[u_1+u_2]z\|^2.$$ 
By Steklov`s Lemma,
$$I_3=\|z_x\|^2\leq a\|z_{xx}\|^2\leq a\|\Delta z\|^2$$ 
and
$$I_4=\frac{1}{2\delta}\|[u_1+u_2]z\|^2\leq \frac{1}{\delta}sup_D|u_1^2+u_2^2|\|z\|^2.$$
Making use of Lemma 4.3, we find that
$$I_5=\frac{1}{\delta}sup_D|u_1^2+u_2^2|\|z\|^2\leq C(\delta)[\|\Delta u_1\|^2+\|\Delta u_2\|^2]\|z\|^2.$$

Substituting $I_1-I_5$ into (5.5) and taking $\delta>0$
sufficiently small, we come to the inequality
$$\dfrac{d}{dt}\|z\|^2(t) \leq C(k)\sum_{i=1}^2[1+\|\Delta u_i\|^2(t)]\|z\|^2(t).$$ Taking into
account that by (4.27) $[1+\|\Delta u_i\|^2(t)] \in L^1(0,T)$ $(i=1,2),$  we
get $\|z\|(t)\equiv 0$ for $a.e.\:t \in (0,T).$ This proves
uniqueness of a regular solution of (3.1)-(3.3) and
completes the proof of Theorem 3.1.

\end{proof}
\begin{rem} Assertions of Theorem 3.1 stay true if equation (3.1) is changed by

$$Lu \equiv u_t + \Delta^2 u+\gamma \Delta u+ \Delta u_x +uu_x-\partial^5_x u=0 \quad \textrm{in} \quad Q_t,$$
where $\gamma$ is a positive real number.
\end{rem}

\section{Decay of Solutions}

In Theorem 3.1, we have proved the existence and uniqueness of global solutions for $t\in (0,T),$ where $T$ is an arbitrary positive number, without  smallness restrictions on initial data and value of an interval $(0,L)$. Nevertheless, we could not  prove decay of solutions because our estimates depended on $T$. Here, we give sufficient conditions which guarantee decay of global solutions as $t\to+\infty.$

{\bf 1. Decay of solutions to the initial-boundary value problem posed on rectangles.}

 First, we prove exponential decay with destabilizing effect of the term $\gamma\Delta u,\;\gamma>0$. Consider the following problem:

\begin{align}
	& Lu \equiv u_t + \Delta^2 u+ \gamma\Delta u+ \Delta u_x +uu_x-\partial^5_x u=0 \quad \textrm{in} \quad Q_t;  \\
	& u|_{\partial D}=u_{yy}(x,0)=u_{yy}(x,B)= u_x(0,y)\notag\\&=u_x(L,y)=u_{xx}(L,y)=0; \\
	& u(x,y,0)=u_0(x,y), \quad (x,y) \in D. 
\end{align}
\begin{teo} \label{tdec}
Let $B,L$  satisfy the following condition: $  \gamma>0;\;b= \gamma\left\{\frac{\pi^2}{B^2}+\frac{\pi^2}{L^2}\right\}^{-1}<1,$ $\theta=1-b>0$. Given
$u_0(x,y)$ subjected to conditions ofTheorem 3.1.   Then for regular solutions of (6.1)-(6.3) the following inequality is true:
\begin{eqnarray*}
&&{\|u\|}^2(t)\leq \|u_0\|^2e^{-\chi t},
\end{eqnarray*}
where\qquad $\chi=2a^2\theta $ and $a=\frac{\pi^2}{B^2}+\frac{\pi^2}{L^2}.$
\end{teo}

\begin{proof}

 Multiplying  (6.1) by $2u$  and integrating the result by parts, we obtain
\begin{eqnarray}
	&&\frac{d}{dt}\|u\|^2(t)-2\gamma\|\nabla u\|^2(t)	+2\|\Delta u\|^2(t)\nonumber\\&& +\int_0^B|D^2_x(0,y,t)^2dy=0.
\end{eqnarray}
By Lemma 4.2, $a\|\nabla u\|^2\leq \|\Delta u\|^2$  that transform (6.4) into the following inequality:
\begin{equation}\frac{d}{dt}\|u\|^2(t)	+2\theta\|\Delta u\|^2(t)\leq 0. \end{equation}
Again by Lemma 4.2, $a^2\|u\|^2\leq\|\Delta u\|^2,$  hence (6.5) becomes
\begin{equation}\frac{d}{dt}\|u\|^2(t)	+2a^2\theta\| u\|^2(t)\leq 0. \end{equation}
Integrating (6.6), we get
$${\|u\|}^2(t)\leq \|u_0\|^2e^{-2a^2\theta t}.$$

This completes the proof of Theorem 6.1.
\end{proof}
\begin{rem} We have proved Theorem 6.1 in  the "bad case" $\gamma>0$ imposing restrictions on dimensions of $D,$ but in the case $\gamma\leq 0$ we will not have any restrictions on dimensions of $D$ or initial data.
\end{rem}
\begin{teo} Let $\gamma$ be nonpositive  real number; $B,L$ are arbitrary positive numbers and $u_0(x,y)$ is a given function
	subjected to conditions of Theorem 3.1. Then global soluions to the problem (6.1)-(6.3) satisfy the following inequality:
$$\|u\|^2(t)\leq \|u_0\|^2e^{-\chi t},$$
where $\chi=2\{\frac{|\gamma|}{a}+1\}a^2.  $
\end{teo}
\begin{proof}	
 Multiplying  (6.1) by $2u$  and integrating the result by parts, we obtain
\begin{eqnarray}
	&&\frac{d}{dt}\|u\|^2(t)+2|\gamma|\|\nabla u\|^2(t)	+2\|\Delta u\|^2(t)\nonumber\\&& +\int_0^B|D^2_x(0,y,t)^2dy=0.
\end{eqnarray}
Since, by Lemma 4.2,  $a\|\nabla u\|^2\geq \| u\|^2$ and  $a^2\|u\|^2\leq\|\Delta u\|^2,$ then (6.7) becomes

\begin{eqnarray}
	&&\frac{d}{dt}\|u\|^2(t)+2\{\frac{|\gamma|}{a}+1\}a^2\|u\|^2(t)\leq 0.
\end{eqnarray}

This implies
$$\|u\|^2(t)\leq \|u_0\|^2e^{-\chi t},$$
where $\chi=2\{\frac{|\gamma|}{a}+1\}a^2.$ This proves Theorem 6.2.
\end{proof}
{\bf 2. Decay of solutions to the initial-boundary value problem posed on the right half-strip.}

Define the right half-strip as follows: $D=(x,y)\in\mathbf{R}^2; x\in(0,+\infty),$\\$y\in(0,B);\;Q_t=D\times (0,t). $
\par In $Q_t$ consider the following problem:
\begin{align}
	& Lu \equiv u_t + \Delta^2 u+ \gamma\Delta u+ \Delta u_x +uu_x-\partial^5_x u=0 \quad \textrm{in} \quad Q_t;  \\
	& u|_{\partial D}=u_{yy}(x,0)=u_{yy}(x,B)= u_x(0,y)=0;\\
	& u(x,y,0)=u_0(x,y), \quad (x,y) \in D. \\&
 \lim_{x\to +\infty} D^{i}_{x}  u_0=0,\;i=0,1,2,3.
\end{align}
\begin{teo} \label{tdec}
	Let  $  \gamma\in (0,\frac{1}{8});\; k\in(0,\frac{1}{4})$ and $ a^2=\frac{\pi^2}{B^2}>1.$\;  Given $u_0(x,y)$  such that
	\begin{eqnarray*}
		&& u_0 \in H^4(D), \quad   \partial_x^5 u_0\;\in L^2(D), \;
		 u_0(0,y)=u_{0x}(0,y)\\&&=u_0(x,0)=u_0(x,B)=u_{0yy}(x,0)=u_{0yy}(x,B)=0, \nonumber\\&&
		J_w=\int_D e^{kx}\big[|\Delta^2 u_0|^2+|\Delta u_0|^2+u^2_0 u^2_{0x}+|D^5_{0x}|^2\big]dxdy<\infty,\end{eqnarray*}
	
	\begin{equation}
	(e^{kx},u_0^2)<\frac{9}{8k}a^2.
	\end{equation}
	
	.   Then for regular solutions of (6.9)-(6.12) the following inequality is true:
	\begin{eqnarray*}
		&&{\|u\|}^2(t)\leq \|u_0\|^2e^{-\chi t},
	\end{eqnarray*}
	where\qquad $\chi=\frac{1}{2}a^2. $ 
\end{teo}

\begin{proof}
	
 Makng use of (4.1), multiplying the $j$-equation of (4.3) by
$e^{kx} g_j(x,t)$, summing over $j=1,..,N$ , we obtain
\begin{equation}
	([u^N_t+ \Delta^2u^N +\gamma \Delta u^N+\partial^3_y u^N_x + u^N u^N_x-\partial^5_x u^N],e^{kx} u^N)(t)=0. 
\end{equation}
Dropping the index $N$ and integrating by parts, we transform (6.14) to the form
\begin{eqnarray}
	&&\frac{d}{dt}(e^{kx},u^2)(t)+(3k-5k^3)(e^{kx},u^2_x)(t)+5k(e^{kx},u_{xx}^2)(t)\nonumber\\
		&& \int_0^B u_{xx}^2(0,y,t)(t)\,dy + k(e^{kx},u^2_y)(t)
	+(k^5-k^3)(e^{kx},u^2)(t)\nonumber\\
	&&+2(e^{kx}u,\Delta^2u)(t)+2\gamma(e^{kx}u,\Delta u)(t)=\frac{2k}{3}(e^{kx},u^3)(t). 
\end{eqnarray}
We estimate
$$I_1=2(\Delta^2 u,e^{kx}u)=2([D^4_y+2D^2_yD^2_x +D^4_x ]u,e^{kx}u)=$$
$$2(e^{kx},|D^2_y u|^2)+2(e^{kx},|D_x D_y u|^2)+(e^{kx},|D^2_x u|^2)-2k^2(e^{kx},u_y^2)-k^3(e^{kx},u^2_x)$$
$$=2(e^{kx},|\Delta u|^2)-2k^2(e^{kx},u_y^2)-k^3(e^{kx},u^2_x).$$
Making use of Lemma 4.2 and Steklov`s inequality, we obtain\\ $ a^2\|u\|^2\leq a^2\|e^\frac{kx}{2}u\|^2\leq (e^{kx},|D^2_yu|^2),$
and $I_1$ becomes
$$I_1=2(\Delta^2 u,e^{kx}u)=2([D^4_y+2D^2_yD^2_x +D^4_x ]u,e^{kx}u)\geq$$
$$2(e^{kx},|D^2_y u|^2)-2k^2(e^{kx},u_y^2)-k^3(e^{kx},u^2_x)$$ 
$$\geq 2a^2(e^{kx},u^2)-2k^2(e^{kx},u_y^2)-k^3(e^{kx},u^2_x).$$ 
Similarly,
$$I_2=2\gamma(e^{kx}u,\Delta u)\leq \gamma[(e^{kx},|\Delta u|^2)+(e^{kx},u^2)].$$

Substituting first $I_2$ and then $I_1$ into (6.15) and taking into account that $k\in (0,1/4),$ and $a>1,$we get

\begin{eqnarray}
	&&\frac{d}{dt}(e^{kx},u^2)(t)+(3k/2)(e^{kx},u^2_x)(t)+5k(e^{kx},u_{xx}^2)(t)\nonumber
	\\
	&& \int_0^B u_{xx}^2(0,y,t)(t)\,dy + (k/2)(e^{kx},u^2_y)(t)
	+(k^5-k^3)(e^{kx},u^2)(t)\nonumber\\
	&&+(2-2\gamma)a^2(e^{kx},u^2)\leq \frac{2k}{3}(e^{kx},u^3)(t). 
\end{eqnarray}
For arbitrary $\epsilon >0$, extending $u$ by zero into the exterior of $D$ and making use of Lemma 2.4 and
(\ref{E1}), we estimate
\begin{eqnarray*}
	&&I_3= \frac{2k}{3}(e^{kx},u^3)(t) \leq
	\dfrac{4k}{3}{\|u\|(t)\|e^{\frac{kx}{2}}u\|}(t){\|
		\nabla(e^{\frac{kx}{2}}u) \|}(t) \\
	&&\leq 2\epsilon(e^{kx},[(u_x^2+u^2_y])(t)+\frac{k^2\epsilon}{2}(e^{kx},u^2)(t)+\frac{k^2}{9\epsilon}(e^{kx},u^2)(t).
\end{eqnarray*}

Substituting $I_3$ into (6.16), we find  		

\begin{eqnarray}
	&&\frac{d}{dt}(e^{kx},u^2)(t)+(3k/2-2\epsilon)(e^{kx},u^2_x)(t)+5k(e^{kx},u_{xx}^2)(t)\nonumber
	\\
	&& \int_0^B u_{xx}^2(0,y,t)(t)\,dy + (k/2-2\epsilon)(e^{kx},u^2_y)(t)
	\nonumber\\
	&&+(2-2\gamma-k^3-k^2\epsilon)a^2(e^{kx},u^2)\leq \frac{k^2}{9\epsilon}(e^{kx},u^2)^2(t). 
\end{eqnarray}

Putting $\gamma=1/8,\;\epsilon=k/8$ e $k=1/4,$ we reduce (6.17) as follows:

\begin{eqnarray}
	&&\frac{d}{dt}(e^{kx},u^2)(t)
	+\frac{1}{2}a^2(e^{kx},u^2)\nonumber\\&& + ([a^2-\frac{8k}{9}(e^{kx},u^2)(t)] )(t)(e^{kx},u^2)(t)\leq 0
\end{eqnarray}
Making use of (6.13) and Lemma 2.5, we come to the inequality
$$\frac{d}{dt}(e^{kx},u^2)(t)
+\frac{1}{2}a^2(e^{kx},u^2)\leq0.$$
Integrating this inequality, we get
\begin{eqnarray*}
	&&{\|u\|}^2(t)\leq \|u_0\|^2e^{-\chi t},
\end{eqnarray*}
where\qquad $\chi=\frac{1}{2}a^2 $ and $a=\frac{\pi^2}{B^2}.$
\par This proves Theorem 6.3.
\end{proof}
The following step is to consider problema (6.9)-(6.12) for $\gamma \geq 0.$
\begin{teo} Let $\gamma$ be an arbitrary nonpositive  real number; $ k\in(0,\frac{1}{4})$;\;$B$ is an arbitrary positive number and $u_0(x,y)$ is a given function
	subjected to conditions of Theorem 6.3. Then global soluions to the problem (6.9)-(6.12) satisfy the following inequality:
	$$\|u\|^2(t)\leq \|u_0\|^2e^{-\chi t},$$
	where $\chi=2\{\frac{|\gamma|}{a}+1\}a^2.  $
\end{teo}
\begin{proof} Acting as by proving (6.17), we obtain
\begin{eqnarray}
	&&\frac{d}{dt}(e^{kx},u^2)(t)+(3k-5k^3)(e^{kx},u^2_x)(t)+5k(e^{kx},u_{xx}^2)(t)\nonumber\\
	&& \int_0^B u_{xx}^2(0,y,t)(t)\,dy + k(e^{kx},u^2_y)(t)
	+(k^5-k^3)(e^{kx},u^2)(t)\nonumber\\&&
	+2(e^{kx}u,\Delta^2u)(t)+2|\gamma|\|\nabla u\|^2(t) \leq
			 2\epsilon(e^{kx},[(u_x^2+u^2_y])(t)\nonumber\\&&+\frac{k^2\epsilon}{2}(e^{kx},u^2)(t)+\frac{k^2}{9\epsilon}(e^{kx},u^2)(t)
\end{eqnarray}
where $\epsilon$ is an arbitrary positive number.
Acting similar to the proof of Theorem 6.3, we get
$$I_1=2(\Delta^2 u,e^{kx}u)\geq 2a^2(e^{kx},u^2)-2k^2(e^{kx},u_y^2)-k^3(e^{kx},u^2_x).$$

Substituting  $I_1$ into (6.19) and taking into account that $k\in (0,1/4),$ and $a>1,$we get

\begin{eqnarray}
	&&\frac{d}{dt}(e^{kx},u^2)(t)+(3k/2)(e^{kx},u^2_x)(t)+5k(e^{kx},u_{xx}^2)(t)\nonumber
	\\
	&& \int_0^B u_{xx}^2(0,y,t)(t)\,dy + (k/2)(e^{kx},u^2_y)(t)
	+(k^5-k^3)(e^{kx},u^2)(t)\nonumber\\
	&&+2(1+\frac{|\gamma|}{a})a^2(e^{kx},u^2) \leq
	2\epsilon(e^{kx},[(u_x^2+u^2_y])(t)\nonumber\\
	&&+\frac{k^2\epsilon}{2}(e^{kx},u^2)(t)+\frac{k^2}{9\epsilon}(e^{kx},u^2)(t).
\end{eqnarray}
Putting $\epsilon, k$ sufficiently small, we find that

\begin{eqnarray}
	&&\frac{d}{dt}(e^{kx},u^2)(t)
+2(1+\frac{|\gamma|}{a})a^2(e^{kx},u^2) \leq 0.
 \end{eqnarray}
Integrating (6.21), we prove Theorem 6.4.
\end{proof}
\begin{rem} We have not proved the existence of regular global solutions in Theorems 6.3, 6.4. It can be done exploiting the approach used in \cite{larkin2}.
\end{rem}
\section{Conclusions}
\par In this work, motivated by \cite{topper}, we study an initial-boundary value problem for the 2D Benney-Lin equation 
adding to their model the Kawahara term $-D_x^5 u$, see \cite{kawa}. We studied stability of the system considering an influense of the term $\gamma\Delta u$. As has been observed in \cite{topper}, this term caused instability while $\gamma$ was positive and implied stability effect while it was negative. 
\par Taking this into account, first we studied in Section 4 the case $\gamma>0$ and proved the existence and uniqueness of global regular solutions on the time interval $(0,T),$ where $T$ was an arbitrary positive number, without restrictions on the size of initial data or on dimensions of an arbitrary bounded domain $D.$  Unfortunately, we could not prove decay of solutions as $t\to +\infty$
because our estimates of solutions depended on $T.$ On the other hand, in Section 6, we have proved  in Theorem 6.1 exponential decay of solutionsfor sufficiently small $\gamma>0$ assuming some restrictions on sizes of a bounded  domain $D.$ In this section, we studied also the case $\gamma\leq 0$ and proved in Theorem 6.2  exponential decay of global solutions without restrictions on sizes of a bounded $D$ and initial data. In Theorem 6.3, we have proved exponential decay of solutions 
defined on the right half-strip with $\gamma>0$ sufficiently small and some restrictions on $u_0$ and $B.$ In Theorem 6.4, exponential decay have been established for global regular solutions defined on the same half-strip while $\gamma\leq 0$
without restrictions on $u_0,B.$ It has been observed that for $\gamma<0$ the term $\gamma \Delta u$ itself supplied stability of the system while the term $\Delta^2u$ served as additional damping. On the other hand, in the case $\gamma=0$ only the term $\Delta^2u$ guaranteed stability .

\end{document}